\documentclass{iopconfser}
\usepackage[utf8]{inputenc}
\usepackage{hyperref}
\usepackage{amsmath}
\usepackage{mathrsfs}
\usepackage{microtype}
\usepackage{amssymb}
\usepackage{subfig}
\usepackage{float}
\usepackage{stmaryrd}
\usepackage{amsthm}
\usepackage[new]{old-arrows}
\usepackage{dsfont}
\usepackage{tikz}
\usetikzlibrary{arrows}

\newtheorem{theor}{Theorem}
\newtheorem{claim}[theor]{Claim}

\theoremstyle{definition}

\newtheorem{proposition}[theor]{Proposition}
\newtheorem{lemma}[theor]{Lemma}
\newtheorem{cor}[theor]{Corollary}
\newtheorem{define}{Definition}

\newtheorem{notation}{Notation}

\newtheorem{example}{Example}
\theoremstyle{remark}
\newtheorem{rem}{Remark}

\def\thickhline{\noalign{\hrule height.8pt}}
\DeclareMathOperator{\emb}{emb}
\DeclareMathOperator{\dif}{d}
\def\oldvec{\mathaccent "017E\relax }
\DeclareMathOperator{\Or}{\mathsf{O\oldvec{r}}}

\newcommand{\im}{\text{im }}

\usepackage{cite}

\begin{document}
\title{Kontsevich graphs act on Nambu--Poisson brackets, \uppercase\expandafter{\romannumeral3\relax}. Uniqueness aspects}

\author{F M Schipper, M S Jagoe Brown and A V Kiselev}

\affil{Bernoulli Institute for Mathematics, Computer Science and Artificial Intelligence, University of Groningen, P.O. Box 407, 9700 AK Groningen, The Netherlands}

\email{f.m.schipper@rug.nl, a.v.kiselev@rug.nl}
\pagestyle{plain}
\begin{abstract}
Kontsevich constructed a map between `good' graph cocycles $\gamma$ and infinitesimal deformations of Poisson bivectors on affine manifolds, that is, Poisson cocycles in the second Lichnerowicz--Poisson cohomology. For the tetrahedral graph cocycle $\gamma_3$ and for the class of Nambu-determinant Poisson bivectors $P$ over $\mathbb{R}^2$, $\mathbb{R}^3$ and $\mathbb{R}^4$, we know the fact of trivialization, $\dot{P}=\llbracket P, \vec{X}^{\gamma_3}_{\text{dim}}\rrbracket$, by using dimension-dependent vector fields $\vec{X}^{\gamma_3}_{\text{dim}}$ expressed by Kontsevich (micro-)\/graphs. We establish that these trivializing vector fields $\vec{X}^{\gamma_3}_{\text{dim}}$ are unique modulo Hamiltonian vector fields $\vec{X}_{H}=\dif_P(H)= \llbracket P, H\rrbracket$, where $\dif_P$ is the Lichnerowicz--Poisson differential and where the Hamiltonians $H$ are also represented by Kontsevich (micro-)\/graphs. However, we find that the choice of Kontsevich (micro-)\/graphs to represent the aforementioned multivectors is not unique.  
\end{abstract}

\section{Introduction}
In 1977, Lichnerowicz introduced a cohomology theory for Poisson manifolds\cite{lichnerowicz1977varietes}. In this theory, the differential is given by $\dif_P=\llbracket P, \ \cdot \ \rrbracket$, where the bracket $\llbracket \ \cdot \ , \
 \cdot \ \rrbracket$ is the Schouten bracket and $P$ is a Poisson bivector. The corresponding cochain complex is given by
\begin{align}\label{full cochain complex}
    0\longrightarrow\mathbb{R}\longhookrightarrow C^{\infty}(M^d) \xrightarrow[]{\dif_P} \mathfrak{X}(M^d)\xrightarrow[]{\dif_P} \mathfrak{X}^2(M^d) \xrightarrow[]{\dif_P} \hdots \xrightarrow[]{\dif_P} \mathfrak{X}^d(M^d) \xrightarrow[]{\ \ \ } 0. \tag{$\star$}
\end{align}  
In 1996, Kontsevich related `good' graph cocycles $\gamma$ in his graph complex $GC$ to infinitesimal deformations of Poisson bivectors $\dot{P}= 
Q^{\gamma}(P)\in \mathfrak{X}^2(M^d_{\textup{aff}})$ (which belong to the kernel of the Poisson differential $d_P$) on an affine Poisson manifold $M^d_{\textup{aff}}$\cite{kontsevich1997formality}. The smallest good graph cocycle $\gamma$ is the tetrahedral graph cocycle $\gamma_3$. We investigate whether the corresponding tetrahedral flow $Q^{\gamma_3}=\Or(\gamma_3)(P)$ is trivial, i.e., whether in addition to being a cocycle, it is also a coboundary. While an immediate thought is to study the trivialization problem on the level of graphs, it is shown that there cannot exist a universal trivializing solution on the level of directed graphs  \cite{kontsevich1997formality,avk}. Instead, we use the morphism $\phi$ to pass from graphs to multivectors \cite{mollie}, and study the trivialization problem $\dot{P}=Q^{\gamma_3}(P)=\llbracket P, \vec{X}^{\gamma_3}\rrbracket$ on the level of multivectors that can be represented by graphs. Let us denote by $\mathfrak{X}^k_{\textup{gra}}(M^d_{\textup{aff}})$ the space of those $k$-vectors on
an affine manifold $M^d_{\textup{aff}}$ which are obtained from graphs. When we restrict the cochain complex \eqref{full cochain complex} to these spaces $\mathfrak{X}^k(M^d_{\textup{aff}})$, we get a subcochain complex for Poisson cohomology \cite{kontsevich1997formality}
\begin{align}\label{cochain complex}
    0\longrightarrow\mathbb{R}\longhookrightarrow C^{\infty}_{\textup{gra}}(M^d_{\textup{aff}}) \xrightarrow[]{\dif_P} \mathfrak{X}_{\textup{gra}}(M^d_{\textup{aff}})\xrightarrow[]{\dif_P} \mathfrak{X}_{\textup{gra}}^2(M^d_{\textup{aff}}) \xrightarrow[]{\dif_P} \hdots \xrightarrow[]{\dif_P} \mathfrak{X}_{\textup{gra}}^d(M^d_{\textup{aff}}) \xrightarrow[]{\ \ \ } 0. \tag{$*$}
\end{align} 
Additionally, we restrict to the class of Nambu-determinant Poisson brackets \cite{buring2022hidden}, and we use dimension specific Kontsevich (micro-)graphs for this class, as introduced in \cite{buring2023tower}. This text is a continuation of \cite{mollie}.

 This text is structured as follows. In section \ref{uniqueness of trivializing} we introduce some notions and phrase the problem. Then we state the main results: the trivializing vector fields of the tetrahedral graph flow over $\mathbb{R}^d$ ($d\leqslant 4$) are unique modulo Hamiltonian vector fields, see section \ref{uniqueness dim 2} for $\mathbb{R}^2$, section \ref{uniqueness dim 3} for $\mathbb{R}^3$ and section \ref{uniqueness dim 4} for $\mathbb{R}^4$. In section \ref{non-uniqueness of graphs} we discuss the non-uniqueness of graphs chosen to represent specific multivectors. (All proofs presented in this text are direct calculations.\footnote{The SageMath package \textsf{gcaops} (\href{https://github.com/rburing/gcaops}{https://github.com/rburing/gcaops}) is used to convert graphs to multivectors, and to solve linear algebraic systems for coefficients of graphs. All the code used for these calculations is attached.})

\section{The trivializing vector fields modulo Hamiltonian vector fields: preliminaries}\label{uniqueness of trivializing}
We consider vector fields $\vec{Y}\in \mathfrak{X}_{\textup{gra}}(\mathbb{R}^d)$ solving the homogeneous equation $\dif_P(\vec{Y})=0$, and we show explicitly that these vector fields are expressed by a (linear combination of) Hamiltonian vector field(s) that we compute in advance. The proof structure is the same for each dimension $d=2,3,4$. 

 Recall that to solve the trivialization problem for the tetrahedral flow, we must find a vector field $\vec{X}^{\gamma_3}_{\textup{dim}}$ satisfying the nonhomogeneous linear algebraic equation 
\begin{align}\label{non hom eq}
    \dot{P} = Q^{\gamma_3}_{\text{dim}}(P)= \llbracket P, \vec{X}^{\gamma_3}_{\textup{dim}}\rrbracket.
\end{align}
As usual, solutions $\vec{Y}_{\textup{dim}}$ to the homogeneous equation,
\begin{align}\label{hom eq}
    \llbracket P, \vec{Y}_{\textup{dim}}\rrbracket=0\in \mathfrak{X}_{\textup{gra}}^2(\mathbb{R}^d),
\end{align}
 give us all the solutions to equation \eqref{non hom eq} via $\vec{X}^{\gamma_3}_{\textup{dim}}+ \vec{Y}_{\textup{dim}}$. In the following sections, we show that the vector fields $\vec{Y}_{\textup{dim}}$ solving the homogeneous system \eqref{hom eq} are Hamiltonian vector fields. 
\begin{define}
    We call $H\in C^{\infty}_{\textup{gra}}(\mathbb{R}^d)$ \textit{Hamiltonians}. Moreover, we call a vector field 
    $\vec{X}_H~\in~\mathfrak{X}_{\textup{gra}}(\mathbb{R}^d)$ \textit{Hamiltonian} if it is in the image of the Lichnerowicz--Poisson differential $\dif_P=\llbracket P,\ \cdot \ \rrbracket$, that is, $\vec{X}_{H}=\dif_P(H)=\llbracket P, H \rrbracket$, for some Hamiltonian $H$.
\end{define}
\begin{notation}
    We will denote a directed edge $(i,j)\in E(\Gamma)$ of a graph $\Gamma$, where $i,j\in V(\Gamma)$, by the shorthand notation $ij$. 
\end{notation}
\begin{define}
    The set of \textit{$d$-dimensional descendants} $(\widehat{\Gamma})_d$ of a two-dimensional Kontsevich graph $\Gamma$ is the collection of all the Nambu micro-graphs obtained from $\Gamma$ by adding $d-2$ Casimir vertices at each Nambu-determinant Poisson structure and redirecting the two original outgoing edges at each Levi-Civita vertex via the Leibniz rule over all the vertices of the targeted Poisson structure(s).
\end{define}
\begin{example}
    Consider the two-dimensional Kontsevich graph encoded\footnote{Edges are issued only from the Levi-Civita vertices. The outgoing edges corresponding to each Levi-Civita vertex are separated by the semicolon ; and encoded by the label of the target vertex. As an extra example, the encoding $[1,2,3; 2,3,4]$ has two Levi-Civita vertices $1,2$, as well as two Casimir vertices $3$ and $4$, and six directed edges $11 \prec 12 \prec 13$, $22\prec23 \prec24$. See also Example \ref{embedding example}. \label{footnote_def_encoding}} by $\Gamma=[1,2;1,2]$. This is a graph on Levi-Civita vertices $1$ and $2$, with directed and ordered edges $11=(1,1)\prec(1,2)=12$, $21=(2,1)\prec(2,2)=22$.
    
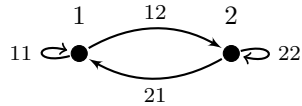
\begin{center}
    \begin{tikzpicture}
    \tikzset{vertex/.style = {shape=circle,draw,minimum size=1.5em}}
    \tikzset{edge/.style = {->,> = latex'}}
     \draw[fill=black] (-1,0) circle (3pt);
     \draw[fill=black] (1,0) circle (3pt);
    \node (1) at (-1,0) {};
    \node (2) at (1,0) {};
    \node at (-1,0.5) {1};
    \node at (1,0.5) {2};
    \draw[thick][edge] (1) to [bend left] node[above] {\footnotesize{12}} (2)  ;
    \draw[thick][edge] (2) to [bend left] node[below] {\footnotesize{21}} (1) {};
    \path[thick] (1) edge [loop left] node[left]{\footnotesize{11}} (1);
    \path[thick] (2) edge [loop right] node[right] {\footnotesize{22}} (2);
    \end{tikzpicture}
    \end{center}
    The three-dimensional Nambu micro-graph descendants of this Kontsevich graph are given by  
    
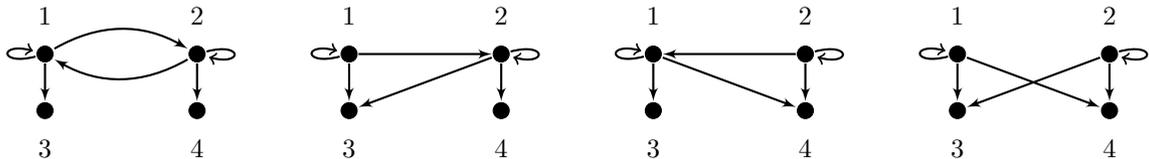
\begin{center}
    \begin{tikzpicture}
    \tikzset{vertex/.style = {shape=circle,draw,minimum size=1.5em}}
    \tikzset{edge/.style = {->,> = latex'}}
     \draw[fill=black] (-5,0) circle (3pt);
     \draw[fill=black] (-7,0) circle (3pt);
     \draw[fill=black] (-7,-0.75) circle (3pt);
     \draw[fill=black] (-5,-0.75) circle (3pt);
     \node (1) at (-7,0) {};
     \node (2) at (-5,0) {};
     \node (3) at (-7,-0.75) {};
     \node (4) at (-5,-0.75) {};
     \node at (-7,0.5) {1};
     \node at (-5,0.5) {2};
     \node at (-7,-1.25) {3};
     \node at (-5,-1.25) {4};
     \draw[thick][edge] (1) to [bend left] node[above] {} (2);
     \draw[thick][edge] (2) to [bend left] node[below] {} (1);
     \draw[thick][edge] (1) to (3) ;
     \draw[thick][edge] (2) to (4);
     \path[thick] (1) edge [loop left] node[left] {} (1);
     \path[thick] (2) edge [loop right] node[right] {} (2);
     \draw[fill=black] (-1,0) circle (3pt);
     \draw[fill=black] (-3,0) circle (3pt);
     \draw[fill=black] (-3,-0.75) circle (3pt);
     \draw[fill=black] (-1,-0.75) circle (3pt);
     \node (5) at (-3,0) {};
     \node (6) at (-1,0) {};
     \node (7) at (-3,-0.75) {};
     \node (8) at (-1,-0.75) {};
     \node at (-3,0.5) {1};
     \node at (-1,0.5) {2};
     \node at (-3,-1.25) {3};
     \node at (-1,-1.25) {4};
     \draw[thick][edge] (5) to (6);
     \draw[thick][edge] (6) to (7);
     \draw[thick][edge] (5) to (7);
     \draw[thick][edge] (6) to (8);
     \path[thick] (5) edge  [loop left] node {} (5);
     \path[thick] (6) edge [loop right] node {} (6);
    \draw[fill=black] (1,0) circle (3pt);
     \draw[fill=black] (3,0) circle (3pt);
     \draw[fill=black] (3,-0.75) circle (3pt);
     \draw[fill=black] (1,-0.75) circle (3pt);
     \node (10) at (3,0) {};
     \node (9) at (1,0) {};
     \node (12) at (3,-0.75) {};
     \node (11) at (1,-0.75) {};
     \node at (1,0.5) {1};
     \node at (3,0.5) {2};
     \node at (1,-1.25) {3};
     \node at (3,-1.25) {4};
     \draw[thick][edge] (9) to (12);
     \draw[thick][edge] (9) to (11);
     \draw[thick][edge] (10) to (9);
     \draw[thick][edge] (10) to (12);
     \path[thick] (9) edge [loop left] node {} (9);
     \path[thick] (10) edge [loop right] node {} (10);
     \draw[fill=black] (5,0) circle (3pt);
     \draw[fill=black] (7,0) circle (3pt);
     \draw[fill=black] (7,-0.75) circle (3pt);
     \draw[fill=black] (5,-0.75) circle (3pt);
     \node (14) at (7,0) {};
     \node (13) at (5,0) {};
     \node (16) at (7,-0.75) {};
     \node (15) at (5,-0.75) {};
     \node at (5,0.5) {1};
     \node at (7,0.5) {2};
     \node at (5,-1.25) {3};
     \node at (7,-1.25) {4};
     \draw[thick][edge] (13) to (15);
     \draw[thick][edge] (13) to (16);
     \draw[thick][edge] (14) to (15);
     \draw[thick][edge] (14) to (16);
     \path[thick] (13) edge [loop left] node {} (13);
     \path[thick] (14) edge [loop right] node {} (14);
    \end{tikzpicture}
    \end{center} 
    where $3$ (respectively $4$) is the Casimir vertex added to Levi-Civita vertex $1$ (respectively $2$). The corresponding encodings are given by
    \begin{align*}
       (\widehat{\Gamma})_{3D}= \big\{[1,2,3;1,2,4], \quad [1,2,3;3,2,4], \quad [1,4,3;1,2,4], \quad [1,4,3;3,2,4]\big\}.
    \end{align*}
\end{example}
\begin{define}
    The \textit{embedding} of a Kontsevich (micro-)graph $\Gamma_{\textup{dim}}$  built from $n$ Nambu-determinant Poisson structures into dimension $\textup{dim}+1$ is the graph $\Gamma_{\textup{dim+1}}=\emb(\Gamma_{\textup{dim}})$ such that to the Levi-Civita vertex of each Nambu-determinant Poisson structure, we add an extra Casimir vertex $a^{d-1}$. The original $d$ outgoing edges of each Levi-Civita vertex keep their order, and the new edge is ordered last. The embedding can often be viewed as a specific type of descendant of a graph. 
\end{define}
\begin{example}\label{embedding example}
    Consider again the two-dimensional Kontsevich graph encoded by $[1,2;1,2]$. The embedding into three dimensions is encoded by $[1,2,3;1,2,4]$, where $3,4$ are the new Casimir vertices. The edges are ordered $11\prec 12\prec  \mathbf{13}$, $21 \prec 22\prec \mathbf{24}$, where $\mathbf{13}$, $\mathbf{24}$ are the new edges.
\end{example}


\section{The trivializing vector fields modulo Hamiltonian vector fields: $\vec{X}^{\gamma_3}_{2D}$ }\label{uniqueness dim 2}
The fact of trivialization of the tetrahedral flow of Poisson bivectors over $\mathbb{R}^2$ has been known since 1996 by M. Kontsevich\cite{kontsevich1997formality} (see also \cite{bouisaghouane2017toy}).
\begin{lemma}\label{2D graphs}
    There are $14$ nonisomorphic Kontsevich graphs on three Levi-Civita vertices and one sink. Explicitly, these $14$ graphs are given by the following encodings.\footnote{Here, $0$ is the sink vertex, the Levi-Civita vertices are given by $1$, $2$, $3$.}  
\footnotesize
\begin{align*}
   \Gamma_1^{2D}&= [0,1;2,3;1,3]& \Gamma_2^{2D}&=[0,1;1,2;1,3] & \Gamma_3^{2D}&=[0,3;2,3;2,3]&  \Gamma_4^{2D}&=[0,3;2,3;1,3] & \Gamma_5^{2D}&=[0,2;2,3;1,3] \\
   \Gamma_6^{2D}&=[0,3;1,2;1,3] & \Gamma_7^{2D}&=[0,3;2,3;1,2] & \Gamma_8^{2D}&=[0,3;1,2;1,2] & \Gamma_9^{2D}&=[0,2;2,3;1,2] &
   \Gamma_{10}^{2D}&=[0,2;1,2;1,2] \\
\Gamma_{11}^{2D}&=[0,1;1,3;1,2] & \Gamma_{12}^{2D}&=[0,3;1,3;1,2]& \Gamma_{13}^{2D}&=[0,1;1,3;2,3]& \Gamma_{14}^{2D}&=[0,1;1,3;1,3]
\end{align*}\normalsize
\end{lemma}
\begin{claim}[See the attached code]
    The images of the $14$ nonisomorphic Kontsevich graphs $\Gamma_1^{2D},...,\Gamma_{14}^{2D}$ under the morhpism $\phi$ from graphs to multivectors satisfy the following linear relations:
    \footnotesize\begin{align}\label{rel 2d}
\phi(\Gamma_1^{2D})&=\phi(\Gamma_5^{2D})=\phi(\Gamma_6^{2D})=-\phi(\Gamma_7^{2D})=\tfrac{1}{2}\phi(\Gamma_8^{2D})=\phi(\Gamma_{12}^{2D}) =\phi(\Gamma_{13}^{2D}), \\
    \phi(\Gamma_2^{2D})&=\phi(\Gamma_4^{2D})=-\phi(\Gamma_9^{2D})=\phi(\Gamma_{11}^{2D}),  &\phi(\Gamma_3^{2D})&=\phi(\Gamma_{10}^{2D})=\phi(\Gamma_{14}^{2D}).\quad \quad \nonumber
\end{align}
\end{claim}
\normalsize This means that in dimension two, when restricting ourselves to the formulas which the $14$ graphs are evaluated into, we get only three linearly independent vector fields $\phi(\Gamma^{2D}_{11})$, $\phi(\Gamma^{2D}_{12})$ and $\phi(\Gamma^{2D}_{3})$: over the first two vector fields we find a unique solution $\vec{X}^{\gamma_3}_{2D}$ to equation \eqref{non hom eq}, while the third vector field solves the homogeneous system \eqref{hom eq}. 
\begin{proposition}[{\cite[Proposition 1]{bouisaghouane2017toy}}]\label{triviality 2d}
   The trivializing vector field $\vec{X}^{\gamma_3}_{2D}$ for the tetrahedral flow of Poisson bivectors $P=\varrho(x,y) \  \partial_{x}\wedge\partial_{y}$ with Cartesian coordinates, up to a normalization constant $\tfrac{1}{8}$, is given by
    \footnotesize \begin{align*}
        \vec{X}^{\gamma_3}_{2D}&= 1\cdot \phi( \Gamma^{2D}_{11}) +2 \cdot \phi(\Gamma^{2D}_{12}) \\
        &= (-2\varrho_{y}(\varrho_{xy})^2 + 2\varrho_{y}\varrho_{xx}\varrho_{yy} + (\varrho_{y})^2\varrho_{xxy} - 2\varrho_{x}\varrho_{y}\varrho_{xyy} + (\varrho_{x})^2\varrho_{yyy})\xi_x \\ 
         &{}\quad + (2\varrho_{x}(\varrho_{xy})^2 - 2\varrho_{x}\varrho_{xx}\varrho_{yy} - (\varrho_{y}^2)\varrho_{xxx} + 2\varrho_{x}\varrho_{y}\varrho_{xxy} - (\varrho_{x})^2\varrho_{xyy})\xi_y. 
     \end{align*}
\normalsize\end{proposition}
\begin{proposition}\label{hom eq 2d}
    On $\mathbb{R}^2$, there is a unique vector field represented by Kontsevich graphs (modulo nonzero constant multiples) $\vec{Y}^{2D}$ such that $\llbracket P, \vec{Y}^{2D}\rrbracket=0$. 
\end{proposition}
\begin{proof}
See the attached code for the computation yielding precisely one vector field solving equation \eqref{hom eq}. Explicitly, the vector field is given by
\footnotesize\begin{align*}
    \phi(\Gamma_3^{2D})=(\varrho\varrho_{yy}\varrho_{xxy} - 2\varrho\varrho_{xy}\varrho_{xyy} + \varrho\varrho_{xx}\varrho_{yyy})\xi_x + (-\varrho\varrho_{yy}\varrho_{xxx} + 2\varrho\varrho_{xy}\varrho_{xxy} - \varrho\varrho_{xx}\varrho_{xyy})\xi_y.
\end{align*}
\normalsize Another direct computation yields that $\dif_P(\phi(\Gamma_3^{2D}))=\llbracket P, \phi(\Gamma_3^{2D})\rrbracket=0$.
\end{proof}
Let us examine the degree of freedom coming from $\phi(\Gamma_{3}^{2D})$. Consider the Hamiltonians we can create from Kontsevich graphs on $\mathbb{R}^2$. Since the vector fields at hand contain three copies of the Poisson structure as vertices, and the Poisson--Lichnerowicz differential $\dif_P=\llbracket P,\ \cdot \ \rrbracket$ adds another Poisson structure, we conclude that our Hamiltonian(s) must contain two Poisson structures. 
\begin{lemma}\label{2d ham}
    There is only one way to create a Kontsevich graph on two Levi-Civita vertices $1$, $2$ and no sink. The encoding for this Hamiltonian graph is given by $\Gamma^{2D}_{H_1}= [1,2;1,2]$.
\end{lemma}
\normalsize 
\begin{notation}
    Put $H^{\textup{dim}}_i=\phi(\Gamma^{\textup{dim}}_{H_i})$; the same notation is used for dimensions three and four.
\end{notation}
\begin{theor}
      On $\mathbb{R}^2$, let $P=\varrho(x,y) \ \partial_x\wedge\partial_y$ be a \textup{(}possibly degenerate\textup{)} Poisson bivector. Consider the complex \eqref{cochain complex} restricted to Hamiltonians on $2$ copies of $P$, vector fields on $3$ copies of $P$ and bivectors on $4$ copies of $P$. We establish that the corresponding homogeneous part of the first Poisson-Lichnerowicz cohomology $H^1_{\textup{gra}}(\mathbb{R}^2)$ is trivial.  
\end{theor}
\begin{proof}[Proof (see the attached code).]
    We write the vector field $\vec{Y}^{2D}$ of Proposition \ref{hom eq 2d} in terms of the Hamiltonian vector field $\dif_P(H_1^{2D})$. By a direct calculation, we establish that
    \footnotesize\begin{align*}
        2\cdot \vec{Y}^{2D}&= 2 \cdot \phi(\Gamma_3^{2D}) =\dif_P(H_1^{2D}), 
    \end{align*}
    \normalsize that is, the degree of freedom is provided by the Hamiltonian shift. It follows immediately that the corresponding homogeneous part of $H_{\text{gra}}^1(\mathbb{R}^2)=\ker \dif_P/\im \dif_P$ is trivial. 
\end{proof}
\begin{cor}
    The trivializing vector field $\vec{X}^{\gamma_3}_{2D}$ of Proposition \ref{triviality 2d} is unique modulo Hamiltonian vector fields. 
\end{cor}

\section{The trivializing vector fields modulo Hamiltonian vector fields: $\vec{X}^{\gamma_3}_{3D}$}\label{uniqueness dim 3}
In dimension three, triviality of the tetrahedral graph cocycle was established in \cite{buring2022hidden, buring2023tower}. Interestingly, we can also find a trivializing vector field over just the descendants $(\widehat{\Gamma}^{2D}_{11})_{3D}$ and $(\widehat{\Gamma}^{2D}_{12})_{3D}$. 
\begin{lemma}[{\cite[Lemma 3]{mollie}}]
    The set $(\widehat{\Gamma}^{2D}_{11})_{3D} \cup (\widehat{\Gamma}^{2D}_{12})_{3D}$
    contains 41 non-isomorphic Nambu micro-graphs.
\end{lemma}

\begin{proposition}\label{triviality 3d}
    For the tetrahedral flow of Poisson bivectors over $\mathbb{R}^3$ , the trivializing vector field $\vec{X}^{\gamma_3}_{3D}$, restricted to vector fields corresponding to the descendants $(\widehat{\Gamma}^{2D}_{11})_{3D}\cup(\widehat{\Gamma}^{2D}_{12})_{3D}$ of the graphs $\Gamma_{11}^{2D}, \Gamma_{12}^{2D}$ of Lemma \ref{2D graphs}, is given by 
    \footnotesize\begin{align*}
        \vec{X}^{\gamma_3}_{3D}&=\ 8\cdot \phi(\Gamma^{3D}_1)+24\cdot\phi(\Gamma^{3D}_4)+8\cdot\phi(\Gamma^{3D}_{7})+24\cdot\phi(\Gamma^{3D}_8)+12\cdot\phi(\Gamma^{3D}_{16})+16\cdot\phi(\Gamma^{3D}_{17})+16\cdot\phi(\Gamma^{3D}_{25})\\
        &+12\cdot\phi(\Gamma^{3D}_{26})+16\cdot\phi(\Gamma^{3D}_{29})+24\cdot\phi(\Gamma^{3D}_{33}).\nonumber
    \end{align*}
\normalsize \end{proposition}
\begin{proof}
    See \cite{mollie} and the attached code.
\end{proof}
 The corresponding encodings\footnote{Here, $0$ is the sink, we have Levi-Civita vertices $1$, $2$ and $3$ with the respective Casimir vertices $4$, $5$ and $6$.} of the $1$-vector graphs appearing in $\vec{X}^{\gamma_3}_{3D}$ are these:
\footnotesize \begin{align*} 
    \Gamma_1^{3D}&=[0, 1, 4; 1, 3, 5; 1, 2, 6]& \Gamma_4^{3D}&=[0, 1, 4; 1, 6, 5; 4, 2, 6] & \Gamma_7^{3D}&=[0, 1, 4; 4, 3, 5; 4, 2, 6]& \Gamma_8^{3D}&=[0, 1, 4; 4, 6, 5; 4, 2, 6]\\
    \Gamma_{16}^{3D}&=[0, 1, 4; 4, 6, 5; 4, 5, 6] & \Gamma_{17}^{3D}&=[0, 2, 4; 1, 3, 5; 1, 2, 6] &\Gamma_{25}^{3D}&=[0, 2, 4; 1, 3, 5; 1, 5, 6]& \Gamma_{26}^{3D}&=[0, 2, 4; 1, 6, 5; 1, 5, 6] \\
     \Gamma_{29}^{3D}&=[0, 2, 4; 4, 3, 5; 1, 5, 6] &  \Gamma_{33}^{3D}&=[0, 5, 4; 1, 3, 5; 1, 2, 6].  
\end{align*}
\normalsize\begin{proposition}\label{hom eq 3d}
 There are three linearly independent vector fields $\vec{Y}_{1}^{3D}$, $\vec{Y}_2^{3D}$ and $\vec{Y}_3^{3D}$ that span the solution space of $\llbracket P, \vec{X}^{3D}\rrbracket=0$ when restricting to solution over linear combinations of the descendants $(\widehat{\Gamma}_{11}^{2D})_{3D}\cup(\widehat{\Gamma}_{12}^{2D})_{3D}$.  
 \end{proposition}
\begin{proof}
    See the attached code for the computation yielding precisely three vector fields solving equation \eqref{hom eq}. Explicitly, these vector fields are (with their encodings found directly below) 
    \footnotesize\begin{align*} 
        \vec{Y}_{1}^{3D}&= 1\cdot \phi(\Gamma^{3D}_2)+1\cdot \phi(\Gamma^{3D}_{18})+1\cdot \phi(\Gamma^{3D}_{34})+1\cdot\phi(\Gamma^{3D}_{41})  \\
        \vec{Y}_{2}^{3D}&= 1\cdot \phi(\Gamma^{3D}_4)+\tfrac{1}{2}\cdot \phi(\Gamma^{3D}_{31})+\tfrac{1}{2}\cdot \phi(\Gamma^{3D}_{45})\\
        \vec{Y}_{3}^{3D}&= 1\cdot \phi(\Gamma^{3D}_{10})+1\cdot \phi(\Gamma^{3D}_{31})+2\cdot \phi(\Gamma^{3D}_{34})+2\cdot\phi(\Gamma^{3D}_{42})+1\cdot\phi(\Gamma^{3D}_{45}).
    \end{align*}
    \end{proof}
    \normalsize  The corresponding encodings\footnote{Here, $0$ is the sink, we have Levi-Civita vertices $1$, $2$ and $3$ with the respective Casimir vertices $4$, $5$ and $6$.} are as follows:
   \footnotesize \begin{align*} 
        \Gamma^{3D}_2&=[0, 1, 4; 1, 6, 5; 1, 2, 6] & \Gamma^{3D}_4&= [0, 1, 4; 1, 6, 5; 4, 2, 6] & \Gamma^{3D}_{10}&=[0, 1, 4; 1, 6, 5; 1, 5, 6]\\
        \Gamma^{3D}_{18}&= [0, 2, 4; 1, 6, 5; 1, 2, 6]& \Gamma^{3D}_{31}&=[0, 2, 4; 4, 3, 5; 4, 5, 6] &\Gamma^{3D}_{34}&=[0, 5, 4; 1, 6, 5; 1, 2, 6] \\
        \Gamma^{3D}_{41}&= [0, 5, 4; 1, 3, 5; 1, 5, 6] &  \Gamma^{3D}_{42}&=[0, 5, 4; 1, 6, 5; 1, 5, 6] &   \Gamma^{3D}_{45}&= [0, 5, 4; 4, 3, 5; 1, 5, 6]
    \end{align*}
\normalsize
Again, we examine these degrees of freedom. 
\begin{lemma}
    There are seven nonisomorphic Nambu micro-graphs on two Levi-Civita vertices $1$, $2$, two corresponding Casimir vertices $3$, $4$ and no sink.
\end{lemma}
 The encodings for these seven Hamiltonians are given directly below.
\footnotesize\begin{align*}\label{ham 3d}
    \Gamma^{3D}_{H_1}&= [2,3,4;1,3,4] & \Gamma^{3D}_{H_2}&=[2,3,4;2,3,4] & \Gamma^{3D}_{H_3}&=[2,3,4;1,2,4] & \Gamma^{3D}_{H_4}&=[1,3,4;1,2,4]\\
     \Gamma^{3D}_{H_5}&= [1,3,4;2,3,4] & \Gamma^{3D}_{H_6}&= [1,2,3;1,2,4]  &\Gamma^{3D}_{H_7}&=[1,2,3;2,3,4]
\end{align*}
\normalsize We detect the following relations (they are explicitly verified in the attached code): 
\footnotesize\begin{align}
    H_1^{3D}&=H_5^{3D}, & H_3^{3D}&=-H_4^{3D}=-H_7^{3D}.
\end{align}\normalsize
\begin{rem}
    The graph $\Gamma^{3D}_{H_6}$ is precisely the three-dimensional embedding of the two-dimensional Hamiltonian $\Gamma^{2D}_{H_1}$ from Lemma \ref{2d ham}. Since we are working only over $(\widehat{\Gamma}^{2D}_{11})_{3D}$ and $(\widehat{\Gamma}^{2D}_{12})_{3D}$, there is no fourth linearly independent vector field $\vec{Y}^{4}_{3D}$ created from linear combinations of vector fields evaluated from $(\widehat{\Gamma}^{2D}_{11})_{3D}$ and $(\widehat{\Gamma}^{2D}_{12})_{3D}$ in dimension three satisfying $\llbracket P, \vec{Y}^{4}_{3D}\rrbracket=0$. When we run the same code over \textit{all} three-dimensional Nambu micro-graphs (with three Levi-Civita vertices, three corresponding Casimir vertices and one sink), we \textit{do} get this fourth vector field $\vec{Y}^{4}_{3D}$ which nontrivially depends on $\dif_P(H_6^{3D})=\llbracket P, H_6^{3D}\rrbracket$.
\end{rem}
\begin{theor}
    On $\mathbb{R}^3$, let $P$ be a \textup{(}degenerate\textup{)} Nambu-determinant Poisson bivector. Consider the complex \eqref{cochain complex} restricted to Hamiltonians on $2$ copies of $P$, vector fields on $3$ copies of $P$, bivectors on $4$ copies of $P$ and trivectors on $5$ copies of $P$. We establish that the corresponding homogeneous part of the first Poisson--Lichnerowicz cohomology $H^1_{\textup{gra}}(\mathbb{R}^3)$ is trivial.   
\end{theor}
\begin{proof}[Proof (see the attached code).]
    We write each of the three vector fields $\vec{Y}_{i}^{3D}$ of Proposition \ref{hom eq 3d} in terms of the Hamiltonian vector fields $\dif_P(H_1^{3D})$, $\dif_P(H_2^{3D})$ and $\dif_P(H_3^{3D})$.
    We compute 
    \footnotesize\begin{align*}
        \vec{Y}_{1}^{3D}&= 1\cdot \dif_P(H^{3D}_3), & \vec{Y}_{2}^{3D}&=\tfrac{1}{4}\cdot \dif_P(H_1^{3D}), & \vec{Y}_{3}^{3D}&= \tfrac{1}{2}\cdot \dif_P(H_1^{3D})-1\cdot \dif_P(H_2^{3D}),         
    \end{align*}\normalsize
    that is, the degrees of freedom are provided by the Hamiltonian shifts. It follows immediately that the corresponding homogeneous part of  $H^1_{\text{gra}}(\mathbb{R}^3)=\ker \dif_P/\im \dif_P$ is trivial.
 \end{proof}
 \begin{cor}
    The trivializing vector field $\vec{X}^{\gamma_3}_{3D}$ over the descendants $(\widehat{\Gamma}^{2D}_{11})_{3D}$ and $(\widehat{\Gamma}^{2D}_{12})_{3D}$ of Proposition \ref{triviality 3d} is unique modulo Hamiltonian vector fields. 
\end{cor}

\section{The trivializing vector fields modulo Hamiltonian vector fields: $\vec{X}^{\gamma_3}_{4D}$}\label{uniqueness dim 4}
In dimension four, triviality of the tetrahedral graph cocycle is established in \cite{mollie}. The trivializing vector field is found again over the descendants of the two-dimensional solution from Proposition~\ref{triviality 2d}, but in addition, we request that the vector field is skew-symmetric with respect to the two Casimirs $a^1$ and $a^2$, see \cite{mollie}. 
\begin{rem}
    As we are working with Nambu-determinant Poisson brackets, we see that the Poisson structure itself is skew-symmetric with respect to the Casimirs $a^1$ and $a^2$. By requesting that our Hamiltonians $H$ are symmetric, we ensure that the Hamiltonian vector fields $\dif_P(H)$ we consider are skew-symmetric with respect to $a^1$ and $a^2$.
\end{rem}
\begin{notation}
    Let us denote by $\phi^{-}(\Gamma_{4D})$ (respectively $\phi^{+}(\Gamma_{4D})$) the skew-symmetrized (respectively symmetrized) multivector obtained\footnote{Consider as an example the four-dimensional graph $\Gamma(a^1,a^2)$ with encoding $[0,1,4,7;1,3,6,9;1,5,8,9]$ where $0$ is the sink, $1,2,3$ are Levi-Civita vertices with corresponding $a^1$ Casimir vertices $4,5,6$ and $a^2$ Casimir vertices $7,8,9$. We simply swap the pairs of Casimirs vertices belonging to each Poisson structure to obtain $\Gamma(a^2,a^1)=[0,1,7,4;1,3,9,6;1,8,5,6]$. Then, $\phi^{-}=\tfrac{1}{2}(\phi(\Gamma(a^1,a^2))-\phi(\Gamma(a^2,a^1)))$.} from the graph $\Gamma_{4D}$ by swapping the Casimirs $a^1$ and $a^2$. We write $(H_{i}^{4D})^{+}$ for the symmetrized Hamiltonian function represented by the graph $\Gamma_{H_i}^{4D}$. 
\end{notation}
\begin{rem}
    In the following, (skew-)symmetry is \textit{always} with respect to the Casimirs $a^1$ and $a^2$.
\end{rem}
\begin{proposition}[{\cite[Proposition 8]{mollie}}]\label{triviality 4d} Over $\mathbb{R}^4$, there exists a skew solution $(\vec{X}^{\gamma_3}_{4D})^-$ solving equation~\eqref{non hom eq}; this trivializing vector field consists of $27$ skew-symmetrized vector fields obtained from the descendants $(\widehat{\Gamma}^{2D}_{11})_{4D}\cup(\widehat{\Gamma}^{2D}_{12})_{4D}$.
\end{proposition}
\begin{proposition}\label{hom eq 4d}
    There are seven linearly independent vector fields $\vec{Y}_{1}^{4D}$, $\vec{Y}_{2}^{4D}$, $\vec{Y}_{3}^{4D}$, $\vec{Y}_{4}^{4D}$, $\vec{Y}_{5}^{4D}$, $\vec{Y}_{6}^{4D}$ and $\vec{Y}_{7}^{4D}$ that span the solution space of $\llbracket P, \vec{X}^{4D}\rrbracket=0$ when restricting to solutions over linear combinations of the skew-symmetrized descendants $(\widehat{\Gamma}^{2D}_{11})_{4D}\cup(\widehat{\Gamma}^{2D}_{12})_{4D}$. 
\end{proposition}
\begin{proof}
    See the attached code for the computation yielding precisely seven skew-symmetric vector fields solving equation \eqref{hom eq}. Explicitly, these vector fields are 
    \footnotesize\begin{align*}
        \vec{Y}^{1}_{4D}&= 1\cdot \phi^{-}(\Gamma_2^{4D})-\tfrac{1}{2}\cdot \phi^{-}(\Gamma_9^{4D})+1\cdot \phi^{-}(\Gamma_{26}^{4D})+\tfrac{1}{2}\cdot \phi^{-}(\Gamma_{33}^{4D})+1\cdot \phi^{-}(\Gamma_{35}^{4D})-1\cdot \phi^{-}(\Gamma_{36}^{4D})+1\cdot \phi^{-}(\Gamma_{40}^{4D})\\
        & \quad -1\cdot \phi^{-}(\Gamma_{41}^{4D})+\tfrac{1}{2}\cdot \phi^{-}(\Gamma_{42}^{4D})+1\cdot \phi^{-}(\Gamma_{48}^{4D})+1\cdot \phi^{-}(\Gamma_{61}^{4D})\\
        \vec{Y}_{4D}^2&=1\cdot \phi^{-}(\Gamma_4^{4D})+\tfrac{1}{2}\cdot \phi^{-}(\Gamma_9^{4D})-1\cdot \phi^{-}(\Gamma_{35}^{4D})+1\cdot \phi^{-}(\Gamma_{36}^{4D})+1\cdot \phi^{-}(\Gamma_{41}^{4D})-\tfrac{1}{2}\cdot \phi^{-}(\Gamma_{42}^{4D}) \\
        \vec{Y}_{4D}^3&= 1\cdot \phi^{-}(\Gamma_{10}^{4D})-1\cdot \phi^{-}(\Gamma_{16}^{4D})+1\cdot \phi^{-}(\Gamma_{18}^{4D})+1\cdot \phi^{-}(\Gamma_{20}^{4D})-\tfrac{1}{2}\cdot \phi^{-}(\Gamma_{24}^{4D})-1\cdot \phi^{-}(\Gamma_{31}^{4D})-1\cdot \phi^{-}(\Gamma_{34}^{4D})\\
        &\quad -1\cdot \phi^{-}(\Gamma_{35}^{4D})+1\cdot \phi^{-}(\Gamma_{36}^{4D})+2\cdot \phi^{-}(\Gamma_{40}^{4D})+1\cdot \phi^{-}(\Gamma_{41}^{4D})- 1\cdot \phi^{-}(\Gamma_{43}^{4D})-1\cdot \phi^{-}(\Gamma_{45}^{4D})+1\cdot \phi^{-}(\Gamma_{46}^{4D})\\
        &\quad -1\cdot \phi^{-}(\Gamma_{47}^{4D})-\tfrac{1}{2}\cdot \phi^{-}(\Gamma_{54}^{4D})
        -1\cdot \phi^{-}(\Gamma_{61}^{4D})+1\cdot \phi^{-}(\Gamma_{63}^{4D})-\tfrac{1}{2}\cdot \phi^{-}(\Gamma_{64}^{4D}) \\
        \vec{Y}_{4D}^4&= 1\cdot \phi^{-}(\Gamma_{12}^{4D})+1\cdot \phi^{-}(\Gamma_{16}^{4D})-1\cdot \phi^{-}(\Gamma_{18}^{4D})-1\cdot \phi^{-}(\Gamma_{20}^{4D})+\tfrac{1}{2}\cdot \phi^{-}(\Gamma_{24}^{4D})+1\cdot \phi^{-}(\Gamma_{31}^{4D})+1\cdot \phi^{-}(\Gamma_{35}^{4D})\\
        &\quad -1\cdot \phi^{-}(\Gamma_{36}^{4D})+1\cdot \phi^{-}(\Gamma_{43}^{4D})-1\cdot \phi^{-}(\Gamma_{46}^{4D})+1\cdot \phi^{-}(\Gamma_{47}^{4D})+\tfrac{1}{2}\cdot \phi^{-}(\Gamma_{54}^{4D})-1\cdot \phi^{-}(\Gamma_{63}^{4D})+\tfrac{1}{2}\cdot \phi^{-}(\Gamma_{64}^{4D})\\
        \vec{Y}_{4D}^5&= 1\cdot \phi^{-}(\Gamma_{14}^{4D})+4\cdot \phi^{-}(\Gamma_{16}^{4D})-4\cdot \phi^{-}(\Gamma_{18}^{4D})+4\cdot \phi^{-}(\Gamma_{31}^{4D})+1\cdot \phi^{-}(\Gamma_{33}^{4D})+2\cdot \phi^{-}(\Gamma_{49}^{4D})+4\cdot \phi^{-}(\Gamma_{62}^{4D})\\
        \vec{Y}_{4D}^6&= 1\cdot \phi^{-}(\Gamma_{15}^{4D})+1\cdot \phi^{-}(\Gamma_{34}^{4D})+2\cdot \phi^{-}(\Gamma_{50}^{4D})+2\cdot \phi^{-}(\Gamma_{62}^{4D})\\
        \vec{Y}_{4D}^7&= 1\cdot \phi^{-}(\Gamma_{22}^{4D})-2\cdot \phi^{-}(\Gamma_{44}^{4D})+1\cdot \phi^{-}(\Gamma_{54}^{4D})+1\cdot \phi^{-}(\Gamma_{64}^{4D}).
    \end{align*}\normalsize
\end{proof}
 The corresponding encodings\footnote{Here, $0$ is the sink, we have Levi-Civita vertices $1$, $2$, $3$ whereas $4$, $5$, $6$ (respectively $7$, $8$, $9$) are the corresponding $a^1$ Casimir vertices (respectively $a^2$ Casimir vertices).} are: 
\footnotesize \begin{align*}
    \Gamma^{4D}_2& =[0, 1, 4, 7; 1, 6, 5, 8; 1, 2, 6, 9] & \Gamma^{4D}_4& =[0, 1, 4, 7; 1, 6, 5, 8; 4, 2, 6, 9] & \Gamma^{4D}_9& = [0, 1, 4, 7; 4, 6, 5, 8; 4, 2, 6, 9]\\
    \Gamma^{4D}_{10}& = [0, 1, 4, 7; 4, 9, 5, 8; 4, 2, 6, 9] & \Gamma^{4D}_{12}& = [0, 1, 4, 7; 4, 6, 5, 8; 7, 2, 6, 9] & \Gamma^{4D}_{14}& =[0, 1, 4, 7; 1, 6, 5, 8; 1, 5, 6, 9]  \\
    \Gamma^{4D}_{15}& = [0, 1, 4, 7; 1, 9, 5, 8; 1, 5, 6, 9] & \Gamma^{4D}_{16}& =[0, 1, 4, 7; 1, 9, 5, 8; 4, 5, 6, 9]  & \Gamma^{4D}_{18}& = [0, 1, 4, 7; 1, 9, 5, 8; 7, 5, 6, 9] \\
    \Gamma^{4D}_{20}& = [0, 1, 4, 7; 4, 9, 5, 8; 4, 5, 6, 9] & \Gamma^{4D}_{22}& = [0, 1, 4, 7; 4, 9, 5, 8; 7, 5, 6, 9] & \Gamma^{4D}_{24}& = [0, 1, 4, 7; 7, 6, 5, 8; 7, 5, 6, 9] \\
    \Gamma^{4D}_{26}& =[0, 2, 4, 7; 1, 6, 5, 8; 1, 2, 6, 9] & \Gamma^{4D}_{31}& =[0, 2, 4, 7; 4, 6, 5, 8; 7, 2, 6, 9] & \Gamma^{4D}_{33}& = [0, 2, 4, 7; 1, 6, 5, 8; 1, 5, 6, 9]\\
    \Gamma^{4D}_{34}& = [0, 2, 4, 7; 1, 9, 5, 8; 1, 5, 6, 9] & \Gamma^{4D}_{35}& = [0, 2, 4, 7; 1, 9, 5, 8; 4, 5, 6, 9]& \Gamma^{4D}_{36}& =[0, 2, 4, 7; 1, 9, 5, 8; 7, 5, 6, 9]\\
    \Gamma^{4D}_{40}& =[0, 5, 4, 7; 1, 9, 5, 8; 1, 2, 6, 9] & \Gamma^{4D}_{41}& =[0, 5, 4, 7; 1, 9, 5, 8; 4, 2, 6, 9]  & \Gamma^{4D}_{42}& = [0, 5, 4, 7; 4, 3, 5, 8; 4, 2, 6, 9]\\
     \Gamma^{4D}_{43}& = [0, 5, 4, 7; 4, 9, 5, 8; 4, 2, 6, 9]& \Gamma^{4D}_{44}& = [0, 5, 4, 7; 7, 9, 5, 8; 4, 2, 6, 9] & \Gamma^{4D}_{45}& = [0, 5, 4, 7; 7, 3, 5, 8; 7, 2, 6, 9]\\
     \Gamma^{4D}_{45}& =[0, 5, 4, 7; 7, 3, 5, 8; 7, 2, 6, 9] &  \Gamma^{4D}_{46}& =[0, 5, 4, 7; 7, 6, 5, 8; 7, 2, 6, 9] &  \Gamma^{4D}_{47}& =[0, 2, 4, 7; 4, 3, 5, 8; 1, 5, 6, 9] \\
    \Gamma^{4D}_{48}& =[0, 5, 4, 7; 1, 3, 5, 8; 1, 5, 6, 9] & \Gamma^{4D}_{49}& =[0, 5, 4, 7; 1, 6, 5, 8; 1, 5, 6, 9] & \Gamma^{4D}_{50}& =[0, 5, 4, 7; 1, 9, 5, 8; 1, 5, 6, 9]\\
    \Gamma^{4D}_{54}& =[0, 5, 4, 7; 4, 9, 5, 8; 7, 5, 6, 9] & \Gamma^{4D}_{61}& = [0, 5, 4, 7; 1, 3, 5, 8; 1, 8, 6, 9] & \Gamma^{4D}_{62}& =[0, 5, 4, 7; 1, 6, 5, 8; 1, 8, 6, 9]\\
    \Gamma^{4D}_{63}& =[0, 5, 4, 7; 7, 3, 5, 8; 7, 8, 6, 9] &  \Gamma^{4D}_{64}& =[0, 5, 4, 7; 7, 6, 5, 8; 7, 8, 6, 9].
\end{align*}
\normalsize 
\begin{lemma}
    There are $21$ nonisomorphic Nambu micro-graphs on two Levi-Civita vertices $1$, $2$, with two corresponding $a^1$ Casimir vertices $3$, $4$ and two corresponding $a^2$ Casimir vertices $5$, $6$ and no sink. The encodings for these $21$ Hamiltonians are given below. 
\footnotesize\begin{align*}
    \Gamma^{4D}_{H_1}&= [1,2,3,5;1,2,4,6]& \Gamma^{4D}_{H_2}&=[1,2,3,5;2,3,4,6] & \Gamma^{4D}_{H_3}&=[1,2,3,5;2,4,5,6] & \Gamma^{4D}_{H_4}&=[1,3,4,5;2,3,4,6]\\
    \Gamma^{4D}_{H_5}&=[1,3,4,5;2,4,5,6] & \Gamma^{4D}_{H_6}&=[1,3,5,6;2,4,5,6] &\Gamma^{4D}_{H_7}&=[1,2,3,5;1,3,4,6]& \Gamma^{4D}_{H_8}&=[1,2,3,5;1,4,5,6]\\
    \Gamma^{4D}_{H_9}&= [1,2,3,5;3,4,5,6] &\Gamma^{4D}_{H_{10}}&=[1,3,4,5;1,3,4,6]& \Gamma^{4D}_{H_{11}}&=[1,3,5,6;1,3,4,6]& \Gamma^{4D}_{H_{12}}&=[1,3,4,5;1,4,5,6]\\
    \Gamma^{4D}_{H_{13}}&=[1,3,5,6;1,4,5,6]& \Gamma^{4D}_{H_{14}}&=[1,3,4,5;3,4,5,6] & \Gamma^{4D}_{H_{15}}&=[1,3,5,6;3,4,5,6] &\Gamma^{4D}_{H_{16}}&=[2,3,4,5;1,3,4,6]\\
    \Gamma^{4D}_{H_{17}}&=[2,3,5,6;1,4,5,6] & \Gamma^{4D}_{H_{18}}&=[2,3,4,5;1,4,5,6] & \Gamma^{4D}_{H_{19}}&=[2,3,4,5;3,4,5,6] & \Gamma^{4D}_{H_{20}}&=[2,3,5,6;3,4,5,6]\\
     \Gamma^{4D}_{H_{21}}&=[3,4,5,6;3,4,5,6]    
\end{align*}
\end{lemma}

\normalsize We detect the following relations (see the attached code). 
\footnotesize\begin{align}\label{ham 4d}
    H_2^{4D}&=-H_7^{4D} &H_4^{4D}&=H_{16}^{4D} &H_6^{4D}&=H_{17}^{4D}& H_{11}^{4D}&=H_{12}^{4D} & H_{15}^{4D}&=H_{20}^{4D}\\
    H_3^{4D}&=-H_8^{4D} &H_5^{4D}&=H_{18}^{4D}& H_9^{4D}&=0 &H_{14}^{4D}&=H_{19}^{4D}    \nonumber   
\end{align}
\normalsize Note that these Hamiltonians are not yet symmetric under $a^1$ and $a^2$. After symmetrizing, we find a maximal linearly independent set consisting of only $(H^{4D}_{1})^{+}$, $(H^{4D}_{2})^{+}$, $(H^{4D}_{4})^{+}$, $(H^{4D}_{5})^{+}$, $(H^{4D}_{10})^{+}$, $(H^{4D}_{11})^{+}$, $(H^{4D}_{14})^{+}$ and $(H^{4D}_{21})^{+}$, see the attached code.
\begin{rem}
    Again, $\Gamma^{4D}_{H_1}$ is precisely the four-dimensional embedding of the two-dimensional Hamiltonian $H^{2D}_1$. As we are only working over $(\widehat{\Gamma}_{11}^{2D})_{4D}\cup(\widehat{\Gamma}_{12}^{2D})_{4D}$, none of the vector fields $\vec{Y}_{i}^{4D}$ will dependent on this Hamiltonian. 
\end{rem}
\begin{theor}
    On $\mathbb{R}^4$, let $P$ be a (degenerate) Nambu-determinant Poisson bivector. Consider the complex \eqref{cochain complex} restricted to symmetric Hamiltonians on $2$ copies of $P$, skew-symmetric vector fields on $3$ copies of $P$, symmetric bivectors on $4$ copies of $P$, etc. We establish that the corresponding homogeneous part of the first Poisson--Lichnerowicz cohomology $H^1_{\textup{gra}}(\mathbb{R}^4)$ is trivial. 
\end{theor}
\begin{proof}[Proof (see the attached code)]
    We write each of the seven skew-symmetric vector fields $\vec{Y}_{i}^{4D}$ of Proposition \ref{hom eq 4d} in terms of the skew-symmetric Hamiltonian vector fields $\dif_P((H_{2}^{4D})^{+})$, $\dif_P((H_{4}^{4D})^{+})$, $\dif_P((H_{5}^{4D})^{+})$, $\dif_P((H_{10}^{4D})^{+})$, $\dif_P((H_{11}^{4D})^{+})$, $\dif_P((H_{14}^{4D})^{+})$ and $\dif_P((H_{21}^{4D})^{+})$. We compute
    \footnotesize\begin{align*}
        \vec{Y}^1_{4D}&= 1\cdot \dif_P((H_2^{4D})^{+})+\tfrac{1}{4}\cdot \dif_P((H_4^{4D})^{+}) & \vec{Y}^5_{4D}&= 1\cdot \dif_P((H^{4D}_{10})^{+})  \\
        \vec{Y}^2_{4D}&= -\tfrac{1}{4}\cdot \dif_P((H_4^{4D})^{+}) &\vec{Y}^6_{4D}&= -1\cdot \dif_P((H^{4D}_{11})^{+}) \\
        \vec{Y}^3_{4D}&= \tfrac{1}{2}\cdot \dif_P((H_{5}^{4D})^{+})-\tfrac{1}{2}\cdot \dif_P((H_{14}^{4D})^{+})-\tfrac{1}{16}\cdot \dif_P((H_{21}^{4D})^{+})& \vec{Y}^7_{4D}&= \tfrac{1}{8}\cdot \dif_P((H_{21}^{4D})^{+}) \\
        \vec{Y}^4_{4D}&= \tfrac{1}{2}\cdot \dif_P((H_{14}^{4D})^{+})+\tfrac{1}{16}\cdot \dif_P((H_{21}^{4D})^{+})
    \end{align*}\normalsize
    that is, the degrees of freedom are provided by the Hamiltonian shifts. It follows immediately that the corresponding homogeneous part of $H^1_{\text{gra}}(\mathbb{R}^4)=\ker \dif_P/\im \dif_P$ is trivial.
\end{proof}
\begin{cor}
    The trivializing skew-symmetric vector field $(\vec{X}^{\gamma_3}_{4D})^-$ over skew-symmetrized vector fields obtained from the descendants $(\widehat{\Gamma}^{2D}_{11})_{4D}$ and $(\widehat{\Gamma}^{2D}_{12})_{4D}$ of Proposition \ref{triviality 4d} is unique modulo skew-symmetric Hamiltonian vector fields. 
\end{cor}

\section{Non-uniqueness of graphs}\label{non-uniqueness of graphs}
\begin{define}
    Two topologically nonisomorphic graphs $\Gamma_1\ncong\Gamma_2$ are called \textit{synonyms} if $\phi(\Gamma_1)=c\cdot \phi(\Gamma_2)$ with $c\in \mathbb{R}\setminus\{0\}$, that is, the two graphs provide the same multivector up to a nonzero constant.  
\end{define}
 We have already seen many synonyms, for example within the three and four dimensional Hamiltonians (equations \eqref{ham 3d}, \eqref{ham 4d}), and the two-dimensional vector fields\footnote{We also have synonyms of vector fields in dimensions $3$ and $4$, see the attached code, but no explicit examples can be given in this text due to volume constraints.} (equation \eqref{rel 2d}). We do not yet understand these synonyms. Two graphs that evaluate to the same multivector in one dimension might not exhibit the same properties in a higher dimension. One of the most clear examples of this is the behaviour of pairs of graphs that give the two-dimensional solution $\vec{X}^{\gamma_3}_{2D}$. Using the relations of equation \eqref{rel 2d}, we can create $28$ pairs in two dimensions such that each pair solves the trivialization problem. But, when we move to dimension three we detect that there exists a solution over the descendants of only $5$ of these $28$ two-dimensional pairs (see the attached code), see table \ref{table 1}.   
\begin{table}[H]\centering
\caption{Does a trivializing vector field exist over the three-dimensional descendants of the trivializing pair $(\Gamma_i^{2D})_{3D},(\Gamma_{j}^{2D})_{3D}$ where $i\in \{2,4,9,11\}$ and $j\in \{1,5,6,7,8,12,13\}$ ?}\label{table 1}
\begin{tabular}{c|c|c|c|c|c|c|c}
\thickhline
  & $(\widehat{\Gamma}_1^{2D})_{3D}$  & $(\widehat{\Gamma}_5^{2D})_{3D}$ & $(\widehat{\Gamma}_6^{2D})_{3D}$ & $(\widehat{\Gamma}_7^{2D})_{3D}$ & $(\widehat{\Gamma}_8^{2D})_{3D}$ & $(\widehat{\Gamma}_{12}^{2D})_{3D}$ & $(\widehat{\Gamma}_{13}^{2D})_{3D}$\\  
\hline
  $(\widehat{\Gamma}_2^{2D})_{3D}$   &  No & No & No & No & \textbf{Yes} & \textbf{Yes} & No \\
   $(\widehat{\Gamma}_4^{2D})_{3D}$  &  No & No & No & No & No & No & No \\
   $(\widehat{\Gamma}_9^{2D})_{3D}$  &  No & No & No & No & No & No & No\\
   $(\widehat{\Gamma}_{11}^{2D})_{3D}$ &  No & No & No & \textbf{Yes}  & \textbf{Yes}& \textbf{Yes} & No\\
\thickhline
\end{tabular}
\end{table}
 Similarly, we can take these $5$ `yes'-pairs over which we find a solution in dimension $3$, and consider their four-dimensional descendants. In this case, we can only find a solution over the descendants of two of these pairs, see table \ref{table 2} (see the attached code).
\begin{table}[H]\centering
\caption{Does a trivializing vector field exist over the four-dimensional descendants of the trivializing pair $(\Gamma_i^{2D})_{4D},(\Gamma_{j}^{2D})_{4D}$ where $i\in \{2,4,9,11\}$ and $j\in \{1,5,6,7,8,12,13\}$?\label{table 2}}
\begin{tabular}{c|c|c|c|c|c|c|c}
\thickhline
  & $(\widehat{\Gamma}_1^{2D})_{4D}$  & $(\widehat{\Gamma}_5^{2D})_{4D}$ & $(\widehat{\Gamma}_6^{2D})_{4D}$ & $(\widehat{\Gamma}_7^{2D})_{4D}$ & $(\widehat{\Gamma}_8^{2D})_{4D}$ & $(\widehat{\Gamma}_{12}^{2D})_{4D}$ & $(\widehat{\Gamma}_{13}^{2D})_{4D}$\\  
\hline
  $(\widehat{\Gamma}_2^{2D})_{4D}$   &  No & No & No & No & \textbf{No} & \textbf{Yes} & No \\
   $(\widehat{\Gamma}_4^{2D})_{4D}$  &  No & No & No & No & No & No & No \\
   $(\widehat{\Gamma}_9^{2D})_{4D}$  &  No & No & No & No & No & No & No\\
   $(\widehat{\Gamma}_{11}^{2D})_{4D}$ &  No & No & No & \textbf{No}  & \textbf{No}& \textbf{Yes} & No\\
\thickhline
\end{tabular}
\end{table}
\section{Conclusion}
    The appearance of synonyms in the trivialization problem makes it difficult to detect patterns in the graphs that show up in the solution (provided both a solution and a pattern exist at all!) and adds an extra barrier in guessing what graphs may appear in the trivializing vector field for a particular dimension.  
\begin{rem}
    We cannot compute in dimension $d\geqslant 5$ because of the time complexity of the system that needs to be solved. Moreover, there is no guarantee that we can find \textit{any} solution over the five-dimensional descendants of the pairs $\Gamma^{2D}_{2}$, $\Gamma^{2D}_{12}$ and $\Gamma^{2D}_{11}$, $\Gamma^{2D}_{12}$, see table \ref{table 2}. Indeed, there is no reason for us to consider \textit{just} the pairs of graphs. We might need to consider linear combinations of more than two graphs as they can still solve the trivialization problem.
\end{rem}

\section*{Acknowledgements}
The authors are grateful to the organizers of the International conference on Integrable Systems and Quantum Symmetries (ISQS 28) for an opportunity to present and discuss new results. The authors thank the Center for Information Technology of the University of Groningen for their support and for providing access to the Hábrók high performance computing cluster. The authors thank the University of Groningen for partial financial support. Lastly, the authors thank R.~Buring for the \textsf{gcaops} software and his instructions on how to work with it. 

\bibliography{references}

\end{document}